\documentclass[12pt,a4paper,reqno]{amsart}
\allowdisplaybreaks
\usepackage{amsmath}
\usepackage{amsfonts}
\usepackage{amssymb,amsthm,amsfonts,amsthm,latexsym,enumerate,url,cases}
\numberwithin{equation}{section}
\usepackage{mathrsfs}
\usepackage{hyperref}
\hypersetup{colorlinks=true,citecolor=blue,linkcolor=blue,urlcolor=blue}
\usepackage{extarrows}
\theoremstyle{plain}
\newtheorem{theorem}{Theorem}

\newtheorem{lemma}{Lemma}
\newtheorem{problem}{Problem}

\newtheorem{proposition}{Proposition}

\theoremstyle{definition}

\usepackage{etoolbox}
\makeatletter
\patchcmd{\@settitle}{\uppercasenonmath\@title}{}{}{}
\patchcmd{\@setauthors}{\MakeUppercase}{}{}{}
\patchcmd{\section}{\scshape}{}{}{}
\makeatother

\usepackage[numbers,sort&compress]{natbib}

\begin{document}

\title
[Note on unique representation bases]
{Note on unique representation bases}

\author
[Y. Ding and J. Wang] 
{Yuchen Ding$^\dag$ and Jie Wang}

\address{(Yuchen Ding$^{1,2}$) $^1$School of Mathematics,  Yangzhou University, Yangzhou 225002, People's Republic of China}
\address{$^2$HUN-REN Alfr\'ed R\'enyi Institute of Mathematics, Budapest, Pf. 127, H-1364 Hungary}
\email{ycding@yzu.edu.cn}

\address{(Jie Wang) School of Mathematical Sciences,  Yangzhou University, Yangzhou 225002, People's Republic of China}
\email{1660777818@qq.com}

\thanks{$^\dag$ Corresponding author}


\begin{abstract}
Answering affirmatively a 2007 problem of Chen, the first author proved that there is a unique representation basis $A$ of $\mathbb{Z}$ and a constant $c>0$ such that
$$
A(-x,x)\ge c\sqrt{x}
$$ 
for infinitely many positive integers $x$, where $A(-x,x)=\big|A\cap[-x, x]\big|$. Let $c_{\mathscr{A}}$ be the least upper bound for such $c$. It was proved in the former article of the first author that $\sqrt{2}/2\le c_{\mathscr{A}}\le \sqrt{2}$. In this note, the prior result is improved to $c_{\mathscr{A}}\ge 1$.

\end{abstract}
\maketitle

\section{Introduction}
Let $\mathbb{N}$ be the set of natural numbers and $A$ a subset of $\mathbb{N}$. 
The Erd\H os-Tur\'an conjecture on additive bases \cite{erdos-turan}, posed in 1941, states that if $R_A(n)\ge 1$ for all sufficiently large $n\in \mathbb{N}$, then $R_A(n)$ cannot be unbounded, where
$$
R_A(n)=\#\big\{(a,a'):n=a+a',a,a'\in A\big\}.
$$
Grekos, Haddad, Helou, and Pihko \cite{GHHP} showed that $R_A(n)\ge 6$ for infinitely many $n$, and Borwein, Choi and Chu \cite{BCC} then improved this to $\limsup_{n\rightarrow\infty}R_A(n)\ge 8$. S\'andor \cite{Sandor} proved that if $\limsup_{n\rightarrow\infty}R_A(n)=K<\infty$, then $\liminf_{n\rightarrow\infty}R_A(n)\le K-2\sqrt{K}+1$.

It turns out that the Erd\H os-Tur\'an conjecture on additive bases does not hold if the natural set $\mathbb{N}$ was replaced by the integer set $\mathbb{Z}$. According to Nathanson \cite{Nathanson}, a subset $A$ of $\mathbb{Z}$ is called a {\it unique representation basis} for the integers if $r_A(n)=1$ for all $n\in \mathbb{Z}$, where
$$
r_A(n)=\#\big\{(a,a'):n=a+a',a\le a',a,a'\in A\big\}.
$$
In \cite{Nathanson}, Nathanson showed that there is a unique representation basis of $\mathbb{Z}$, and furthermore it can be constructed to be arbitrarily sparse. Nathanson also proved that there is a unique representation basis $A$ for the integers such that
$$
A(-x,x):=\big|A\cap [-x,x]\big|\ge (2/\log 5)\log x+0.63.
$$
Nathanson then asked the following interesting problem.

\begin{problem}[Nathanson]\label{problem1}
Does there exists a number $\theta<1/2$ so that $A(-x,x)\le x^\theta$ for every unique representation basis $A$ and for all sufficiently large $x$? 
\end{problem}

Later, Problem \ref{problem1} was answered negatively by Chen \cite{chen} who proved that for any $\varepsilon>0$ there is a unique representation basis $A$ for the integers such that $A(-x,x)\ge x^{1/2-\varepsilon}$ for infinitely many positive integers $x$.
This leads him to the following two problems. 

\begin{problem}[Chen]\label{problem2}
Does there exist a real number $c>0$ and a unique representation basis $A$ such that
$A(-x,x)\ge c\sqrt{x}$ for infinitely many positive integers $x$?
\end{problem}

\begin{problem}[Chen]\label{problem3}
Does there exist a real number $c>0$ and a unique representation basis $A$ such that
$A(-x,x)\ge c\sqrt{x}$ for all real numbers $x\ge 1$?
\end{problem}

In a recent article, the first author \cite{Ding} showed that the answer to Problem \ref{problem2} is affirmative while the answer to Problem \ref{problem3} is negative. For Problem \ref{problem2}, the first author proved that there is a unique representation basis $A$ such that 
\begin{align}\label{eq-0204-1}
A(-x,x)\ge \big(\sqrt{2}/2-\varepsilon\big)\sqrt{x}
\end{align}
for infinitely many positive integers $x$, where $\varepsilon>0$ is arbitrarily small. 

As in \cite{Ding}, let $\mathscr{A}$ be set of all unique representation bases $A$ such that there exists a real number $c>0$ so that $A(-x,x)>c x^{1/2}$ for infinitely many positive integers $x$. For any $A\in \mathscr{A}$, let 
\begin{align}\label{eq-0204-2}
c_A:=\limsup_{x\rightarrow\infty}\frac{A(-x,x)}{\sqrt{x}} \quad \text{and} \quad c_{\mathscr{A}}:=\sup_{A\in \mathscr{A}}c_A.
\end{align}
By (\ref{eq-0204-1}) the set $\mathscr{A}$ is nonempty, and furthermore
$
c_{\mathscr{A}}\ge \sqrt{2}/2.
$ 
For any unique representation basis $A$, it is clear that the elements of 
$$
\big(A\cap[-x, x]\big)+(x+1):=\{a+x+1: a\in A, -x\le a\le x\}
$$
form a Sidon set in $[1, 2x+1]$. Hence, we have $A(-x,x)\le \big(\sqrt{2}+o(1)\big)\sqrt{x}$ (see Lemma \ref{lem1} below). Therefore, so far we got
\begin{align*}
\sqrt{2}/2\le c_{\mathscr{A}}\le \sqrt{2}.
\end{align*}
Naturally, one would like to ask `{\it what is the exact value of $c_{\mathscr{A}}$}'? 

In this note, the lower bound $\sqrt{2}/2$ of $c_{\mathscr{A}}$ is improved.

\begin{theorem}\label{thm:1}
Let $c_{\mathscr{A}}$ be defined in $(\ref{eq-0204-2})$. Then we have $c_{\mathscr{A}}\ge 1$.
\end{theorem}

We {\it conjecture} $c_{\mathscr{A}}=\sqrt{2}$ due to some considerations from Sidon sets. 

\section{Proofs}
The proof of Theorem \ref{thm:1} makes use of the results from Sidon sets. 
An integer set 
$$
S=\{s_1<s_2<\cdot\cdot\cdot<s_t\}
$$ 
is called a {\it Sidon set} if all the sums 
$
s_i+s_j~(i\le j)
$
are different. Let $f_2(N)$ be the largest cardinality of a Sidon set in $\big\{1,2,\cdots, N\big\}$.
We need a well-known result on the asymptotic formula of $f_2(N)$ (see e.g., \cite[Theorem 7, page 88]{Sidon} or \cite{O'}).
\begin{proposition}\label{prop1}
As $N\rightarrow\infty$, we have $f_2(N)\sim \sqrt{N}$.
\end{proposition}

Different to the earlier proof by the first author \cite{Ding}, we do not rely only on the asymptotic formula of $f_2(N)$, but also adopt a classical result of Bose \cite{Bose}, see also \cite[Theorem 3, page 81]{Sidon}.
\begin{proposition}\label{prop2}
Let $p$ be a prime. Then, there are at least $p$ elements $\{g_1,g_2,\cdots, g_p\}$ in $\big\{1,2,\cdots,p^2-1\big\}$ such that all the sums
$$
g_i+g_j \quad (1\le i\le j\le p)
$$ 
are different modulo $p^2-1$.
\end{proposition}
Historically, theorems of the type of Proposition \ref{prop2} were initially proved by Singer \cite{Singer} which were later used to obtain the asymptotic formula of $f_2(N)$( see e.g., \cite{Chowla,Erdos-2}).
The following lemma, which is a consequence of Proposition \ref{prop1} and Proposition \ref{prop2}, plays the key role in the improvement of $c_{\mathscr{A}}$.

\begin{lemma}\label{lem1}
Let $\varepsilon>0$ be an arbitrarily small given number. Then, for any sufficiently large prime $p$ there is a Sidon set 
$$
\mathscr{S}\subset \bigg(\left[-p^2, \left(-\frac{1}{2}-\frac{\varepsilon^2}{8}\right)p^2\right]\bigcup \left[\left(\frac{1}{2}+\frac{\varepsilon^2}{8}\right)p^2, p^2\right]\bigg)
$$ 
such that $\big|\mathscr{S}\big|\ge \left(1-\frac{3\varepsilon}{4}\right)p$.
\end{lemma}
\begin{proof}
Let $p$ be a sufficiently large prime. By Proposition \ref{prop2} there are $p$ integers
$$
s_1<\cdots<s_\ell\le \left(\frac{1}{2}-\frac{\varepsilon^2}{8}\right)p^2<s_{\ell+1}<\cdots<s_y<\left(\frac{1}{2}+\frac{\varepsilon^2}{8}\right)p^2\le s_{y+1}<\cdots<s_p
$$
in the interval $\big[1, p^2\big]$ such that 
all sums 
$
s_i+s_j\ (i\le j)
$
are different modulo $p^2-1$. Let
$$
\mathscr{S}=\Big\{s_1-(p^2-1),s_2-(p^2-1),\cdots, s_{\ell-1}-(p^2-1),s_{y+1}, s_{y+2},\cdots, s_p\Big\}.
$$
Then, clearly we have
$$
\mathscr{S}\subset \bigg(\left[-p^2, \left(-\frac{1}{2}-\frac{\varepsilon^2}{8}\right)p^2\right]\bigcup \left[\left(\frac{1}{2}+\frac{\varepsilon^2}{8}\right)p^2, p^2\right]\bigg)
$$ 
Note that $\big\{s_{\ell+1}, s_{\ell+2}, \cdots, s_y\big\}$ is a Sidon set in the interval 
$$
\left[\left(\frac{1}{2}-\frac{\varepsilon^2}{8}\right)p^2, \left(\frac{1}{2}+\frac{\varepsilon^2}{8}\right)p^2\right],
$$ 
which is of length $\frac{\varepsilon^2}{4}p^2$. Thus, by Proposition \ref{prop1} we have
\begin{align}\label{eq-lem-2}
y-\ell\le \big(1+o(1)\big)\sqrt{\frac{\varepsilon^2}{4}p^2}\le \frac{5\varepsilon p}{8}.
\end{align}
Hence, by (\ref{eq-lem-2}) we have
$$
\big|\mathscr{S}\big|=p-(y-\ell)-1\ge p-\frac{5\varepsilon p}{8}-1\ge \left(1-\frac{3\varepsilon}{4}\right)p.
$$

It remains to show that $\mathscr{S}$ is a Sidon set. Let
\begin{align*}
\widetilde{s_i}=
\begin{cases}
s_i-(p^2-1), & \text{if~}1\le i\le \ell-1,\\
s_i, & \text{if~}y+1\le i\le t.
\end{cases}
\end{align*}
Since 
$
s_i+s_j\ (i\le j)
$
are different modulo $p^2-1$, we know that 
$
\widetilde{s_i}+\widetilde{s_j}\ (i\le j)
$
are also different modulo $p^2-1$, which implies that $\mathscr{S}=\big\{\widetilde{s_i}\big\}$ is a Sidon set.
\end{proof}

Now, we turn to the proof of Theorem \ref{thm:1}.
\begin{proof}[Proof of Theorem \ref{thm:1}]
Let $\varepsilon>0$ be arbitrarily small. It suffices to construct a unique representation basis $A$ such that for infinitely many positive integers $x$ we have
\begin{align}\label{eq-0204-4}
A(-x,x)\ge \big(1-\varepsilon\big)\sqrt{x}.
\end{align}
For any integer set $\mathcal{S}$ and integer $s$, define
$$
2\mathcal{S}:=\big\{s_1+s_2: s_1, s_2\in \mathcal{S}\big\}, \quad \mathcal{S}+s:=\{s_1+s: s_1\in \mathcal{S}\}.
$$
For convenience, we follow some notations adopted by \cite{Ding}.
The construction will be done via inductive process. 
We will construct a series of subsets
$$
A_1\subset A_2\subset\cdots \subset A_i\subset\cdots
$$
and a series of positive integers
$$
x_1<x_2<\cdots<x_i<\cdots
$$
such that for any positive integer $h$ we have the following four properties:

I. $r_{A_h}(n)\le 1$ for any $n\in \mathbb{Z}$,

II. $r_{A_{2h}}(n)=1$ for any $n\in \mathbb{Z}$ with $|n|\le h$,

III. $A_{2h-1}(-x_h,x_h)\ge \big(1-\varepsilon\big)\sqrt{x_h},$

IV. $0\not\in A_h$.

\medskip

First, we set $A_1=\{-1,1\}$ and $x_1=1$. It can be easily checked that inductive hypotheses I, III, IV satisfy with $A_1$. Assume that the sets $A_1\subset A_2\subset\cdots \subset A_{2h-1}$ and integers $x_1<x_2<\cdots<x_h$ satisfying I to IV had already been constructed. We are going to construct $A_{2h},\ A_{2h+1}$, and $x_{h+1}$. Suppose that $m$ is the integer with minimum absolute value such that $r_{A_{2h-1}}(m)=0$. Then by $A_{2h-2}\subset A_{2h-1}$ and inductive hypothesis II, we have $|m|\ge h$. 
Let $a^*_{2h-1}$ be the element of $A_{2h-1}$ with maximum absolute value. Set
$
b=4|a^*_{2h-1}|+|m|
$
and
$$
B=A_{2h-1}\cup\big\{-b,b+m\big\}.
$$
Then $0\not\in B$ and $r_B(m)\ge 1$ since $m=-b+(b+m)$. It can be checked that 
$$
2A_{2h-1}, \quad A_{2h-1}+(-b), \quad  A_{2h-1}+(b+m), \quad  \big\{-2b, m, 2b+2m\big\}
$$
are four disjoint sets, which implies that $r_B(n)\le 1$ for any $n\in \mathbb{Z}$. Hence, we have $r_B(m)= 1$. Furthermore, if $r_B(-m)=0$ we let $\widetilde{b}=4b+5|m|$ and
$$
\widetilde{B}=B\cup\big\{-\widetilde{b},\widetilde{b}-m\big\}.
$$
Then, we have $0\not\in \widetilde{B},\ r_{\widetilde{B}}(-m)=1$, and $r_{\widetilde{B}}(n)\le 1$ for any $n\in \mathbb{Z}$ via similar arguments.
Let $A_{2h}=B$ if $r_B(-m)=1$, and let $A_{2h}=\widetilde{B}$ otherwise.
Clearly, by the above constructions we know that $A_{2h}$ satisfies inductive hypotheses I, II and IV. We are left over to construct a set $A_{2h+1}$ and an integer $x_{h+1}>x_h$ satisfying $A_{2h}\subset A_{2h+1}$ and inductive hypotheses I, III, IV. 

Let $a^*_{2h}$ be the element of $A_{2h}$ with maximum absolute value. Let $p$ be a large integer (in terms of $|a^*_{2h}|$, $\varepsilon$, and $x_h$) to be decided later. Then by Lemma \ref{lem1} there is a Sidon set $\mathscr{S}$ such that all of whose elements belong to the interval
$$
\left[-p^2,\left(-\frac{1}{2}-\frac{\varepsilon^2}{8}\right)p^2\right]\bigcup \left[\left(\frac{1}{2}+\frac{\varepsilon^2}{8}\right)p^2, p^2\right]
$$ 
Additionally, we have 
\begin{align}\label{eq-thm-1}
\big|\mathscr{S}\big|\ge \left(1-\frac{3\varepsilon}{4}\right)p.
\end{align}
To the aim of our purpose, we will delete a few elements from $\mathscr{S}$. Let 
$$
A_{2h}-A_{2h}:=\big\{a-a': a, a'\in A_{2h}\big\}.
$$
Then, there are at most $2|A_{2h}|^2\le 2\cdot\big(2a^*_{2h}\big)^2=8\big(a^*_{2h}\big)^2$ elements in the set 
$$
(A_{2h}-A_{2h})\cup 2A_{2h}.
$$
For any nonnegative element $a_1\in A_{2h}-A_{2h}$, there is at most one pair $(s,s')\in \mathscr{S}^2$ with $s_1\neq s_1'$ such that $s_1-s_1'=a_1$ since $\mathscr{S}$ is a Sidon set. Also, for any $a_2\in 2A_{2h}$, there is at most one pair $(s_2,s_2')\in \mathscr{S}^2$ with $s_2\neq s_2'$ such that $s_2+s_2'=a_2$ since $\mathscr{S}$ is a Sidon set.
We remove all these pairs $(s_1,s_1'), (s_2, s_2')$ to form a smaller Sidon set $\widetilde{\mathscr{S}}$. Clearly, by (\ref{eq-thm-1}) and the discussions above we have
$$
\big|\widetilde{\mathscr{S}}\big|\ge \left(1-\frac{3\varepsilon}{4}\right)p-16\big(a^*_{2h}\big)^2\ge \left(1-\varepsilon\right)p,
$$
provided that $p\ge \frac{64(a^*_{2h})^2}{\varepsilon}$. We choose $p$ to be a sufficiently large prime greater than 
\begin{align}\label{eq-thm-2}
\max\left\{\frac{64(a^*_{2h})^2}{\varepsilon}, x_h\right\}
\end{align}
and $x_{h+1}=p^2$. Then, clearly $x_{h}<x_{h+1}$. Furthermore, we take
$$
A_{2h+1}=A_{2h}\cup \widetilde{\mathscr{S}}.
$$
Then, we have $A_{2h}\subset A_{2h+1}$. It can be seen that $0\not\in A_{2h+1}$ by the construction, and
$$
A_{2h+1}(-x_{h+1},x_{h+1})\ge \big|\widetilde{\mathscr{S}}\big|\ge \left(1-\varepsilon\right)p=\left(1-\varepsilon\right)\sqrt{x_{h+1}},
$$
satisfying inductive hypotheses III and IV. It remains to show that $A_{2h+1}$ satisfies inductive hypothesis I. Let $y_1, y_2, z_1$, and $z_2$ be four elements of $A_{2h+1}$ with
\begin{align}\label{eq-thm-3}
y_1+y_2=z_1+z_2.
\end{align}
It suffices to prove that $\{y_1,y_2\}=\{z_1,z_2\}$. We separate the arguments into a few cases.

{\it Case 1.} All of $y_1, y_2, z_1, z_2$ are all elements of $\widetilde{\mathscr{S}}$. In this case, the desired result is trivial since $\widetilde{\mathscr{S}}$ is a Sidon set.

{\it Case 2.} All of $y_1, y_2, z_1, z_2$ are all elements of $A_{2h}$. In this case, the desired result follows from inductive hypothesis I of $A_{2h}$.

{\it Case 3.} Exactly one element of $y_1, y_2, z_1, z_2$ belongs to $\widetilde{\mathscr{S}}$. In this case, we assume, without loss of generality, that $y_1\in \widetilde{\mathscr{S}}$. Then by (\ref{eq-thm-3}) we have
$$
y_1=z_1+z_2-y_2.
$$
Note that 
$|z_1+z_2-y_2|\le 3|a^*_{2h}|$ while 
$$
|y_1|\ge \frac{1}{2}p^2> 8|a^*_{2h}|^4>3|a^*_{2h}|
$$
from (\ref{eq-thm-2}), which is certainly a contradiction.

{\it Case 4.} Exactly two elements of $y_1, y_2, z_1, z_2$ belong to $\widetilde{\mathscr{S}}$. In this case, the desired result is true due to the construction of $\widetilde{\mathscr{S}}$ from $\mathscr{S}$.

{\it Case 5.} Exactly three elements of $y_1, y_2, z_1, z_2$ belong to $\widetilde{\mathscr{S}}$. Without loss of generality, we assume that $y_1\in A_{2h}$ and $y_2, z_1, z_2\in \widetilde{\mathscr{S}}$. Since $\widetilde{\mathscr{S}}\subset \mathscr{S}$, we have
\begin{align}\label{eq-thm-4}
y_2, z_1, z_2\in \left[-p^2,\left(-\frac{1}{2}-\frac{\varepsilon^2}{8}\right)p^2\right]\bigcup \left[\left(\frac{1}{2}+\frac{\varepsilon^2}{8}\right)p^2, p^2\right].
\end{align}
It then follows from (\ref{eq-thm-2}) and (\ref{eq-thm-4}) that
$$
|z_1+z_2-y_2|\ge \frac{\varepsilon^2}{4}p^2\ge 4|a^*_{2h}|^4>|a^*_{2h}|\ge |y_1|,
$$
which is a contradiction.

Inductive hypothesis I of $A_{2h+1}$ is now proved by {\it Case 1} to {\it Case 5}. 
Finally, we let
$$
A=\bigcup_{h=1}^{\infty}A_h.
$$
Then, $r_A(n)=1$ for any $n\in \mathbb{Z}$ by inductive hypotheses I and II, implying that $A$ is a unique representation basis for the integers. Moreover, inductive hypothesis III implies
$$
A(-x_h,x_h)\ge A_{2h-1}(-x_h,x_h)\ge \big(1-\varepsilon\big)\sqrt{x_h}
$$
for any $x_h$, proving our theorem.
\end{proof}

\end{document}